\documentclass[reqno,a4paper,11pt]{amsart}
\usepackage{amsmath,amsthm,amssymb,color,graphicx,bm,mathtools}
\usepackage[left=2.5cm,right=2.5cm,top=2cm,bottom=2cm,bindingoffset=0cm]{geometry}
\usepackage{algorithm}
\usepackage{algorithmic,comment}
\usepackage{microtype}
\usepackage{emptypage}
\usepackage{indentfirst}
\usepackage{enumitem}
\usepackage{lmodern}
\usepackage{multicol}
\usepackage[bookmarks,colorlinks=true,linkcolor=blue]{hyperref}
\usepackage[all, pdf]{xy}
\usepackage{subfig}
\usepackage{enumitem}
\usepackage{color}
\usepackage{graphicx,booktabs}

\newcommand{\numberset}{\mathbb}
\newcommand{\N}{\numberset{N}} 
\newcommand{\R}{\numberset{R}} 
\newcommand{\Sf}{\numberset{S}} 

\newtheorem{theorem}{Theorem}

\numberwithin{equation}{section}

\begin{document}

\author[D.\,Buoso]{Davide Buoso}
\author[R.\,Molinarolo]{Riccardo Molinarolo}

\address[D.\,Buoso]{Dipartimento per lo Sviluppo Sostenibile e la Transizione Ecologica, Universit\`a degli Studi del Piemonte Orientale ``Amedeo Avogadro'', Piazza Sant'Eusebio 5, 13100, Vercelli, Italy}
\email{davide.buoso@uniupo.it}
\address[R.\,Molinarolo]{Dipartimento per lo Sviluppo Sostenibile e la Transizione Ecologica, Universit\`a degli Studi del Piemonte Orientale ``Amedeo Avogadro'', Piazza Sant'Eusebio 5, 13100, Vercelli, Italy}
\email{riccardo.molinarolo@uniupo.it}

\title[On the eigenvalues of the biharmonic operator on annuli]{On the eigenvalues of the biharmonic operator on annuli}

\begin{abstract}
We show that the fundamental tone of the bilaplacian with Dirichlet or Navier boundary conditions on radially symmetric domains is always simple in dimension $N\ge3$. In dimension $N\ge2$ we show that it is simple if the inner radius is big enough.

\end{abstract}

\maketitle


\bigskip




\section{Introduction}
One of the most used properties of the first eigenvalue of the Dirichlet Laplacian
\begin{equation}\label{laplacian}
\begin{cases}
    \Delta u = \lambda u, & \text{in }\Omega,
    \\
    u=0, & \text{on }\partial\Omega,
\end{cases}
\end{equation}
on a bounded domain $\Omega\subseteq\mathbb R^N$ is that the first eigenvalue is simple and every associated eigenfunction does not change sign in $\Omega$. This can be traced all the way to the validity of the maximum principle, which in turn allows for the Krein-Rutman argument to work (see e.g., \cite[Theorem 3.3]{ggs}). Higher order eigenvalue problems, on the other hand, present a different situation. For instance, in the case of the
Dirichlet bilaplacian
\begin{equation}\label{eq:strongdirev}
\begin{cases}
    \Delta^2 u = \lambda u, & \text{in }\Omega,
    \\
    u=\partial_\nu u =0, & \text{on }\partial\Omega,
\end{cases}
\end{equation}
where $\nu$ denotes the unit outer normal, the first eigenfunction might be sign-changing. This is the case when $\Omega$ is a square \cite{cof} or an elongated ellipse \cite{gar}. More generally, if the domain has corners, then all the eigenfunctions oscillate in the corner (see \cite{cofduf80, kozkonmaz}). This is, to some extent, due to the lack of a maximum principle: in fact, for domains where the Green function is known to be positive, then a Krein-Rutman-type argument is valid implying simplicity of the first eigenvalue and positivity of the corresponding eigenfunction (see \cite[Theorem 3.7]{ggs}).

In particular, in \cite{cofdufsha} the authors show that the first eigenvalue of the Dirichlet bilaplacian \eqref{eq:strongdirev} in an annulus in $\mathbb R^2$ is double when the inner radius is sufficiently small, while it is simple when the inner radius is sufficiently big, with the threshold radius providing an example of a domain with a triple first eigenvalue. 
It must be noted that, more recently, a similar study was performed in \cite{buopar} for the eigenvalues of the bilaplacian buckling problem, where instead the first eigenvalue was shown to be always multiple. 

The aim of the present paper is to extend the study initiated in \cite{cofdufsha} to the case of higher dimension. We consider balls, punctured balls, and spherical shells in $\mathbb R^N$ for any $N\ge 3$ and, after computing somewhat explicitly the eigenvalues and the eigenfunctions, we show that, differently from the two-dimensional situation, the first eigenvalue is always simple. The proof is an extension of the argument already introduced in \cite{cofdufsha}, and in fact we are able to recover the corresponding result in dimension two.

The same techniques allow us also to treat the Navier eigenvalue problem
\begin{equation}\label{eq:strongnavev}
\begin{cases}
    \Delta^2 u = \lambda u, & \text{in } \Omega,
    \\
    (1+\sigma)\partial_\nu^2 u + \sigma \Delta u = 0, & \text{on } \partial\Omega,
    \\
    u = 0, & \text{on } \partial\Omega,
\end{cases}
\end{equation}
where $\sigma \in \left(-\frac{1}{N-1}, 1\right)$ is the so-called Poisson coefficient. For this problem we recover a result similar to the Dirichlet bilaplacian, i.e., the first eigenvalue can be multiple for annuli only in dimension two, while it is always simple for $N\ge 3$.

Our analysis starts with the computation of eigenfunctions and eigenvalues of problems \eqref{eq:strongdirev} and \eqref{eq:strongnavev} on balls, punctured balls, and spherical shells. The general idea is classical and begins with the observation that all eigenfunctions must be in the form of a radial function multiplied by a spherical harmonic function (cf.\ \cite{bf25, chas}). In the ball it is widely known that the radial part of the eigenfunction is a combination of Bessel functions and modified Bessel functions of the first kind, while for spherical shells all types of Bessel functions must be considered. The case of the punctured ball is instead trickier: while for high dimension it is no different from the ball, in low dimension the difference is significant, and this is linked with the behavior of the $H^2$-capacity (cf.\ \cite{bfgm}).

We then move to the identification of the fundamental tone for all these domains. The basic idea is that we can separate the eigenbranches by mean of the associated spherical harmonic degree, and in particular this translates in a simplified form for the Rayleigh quotient which allow restrict the search of the fundamental tone to a small number of eigenbranches. In particular, if the dimension is at least three the smallest eigenvalue is necessarily associated with a radial eigenfunction only, implying in turn its simplicity. 

In dimension two, instead, we see the possibility for the fundamental tone to be associated with an eigenfunction which is not radial. Indeed, this was already known for the Dirichlet bilaplacian \cite{cofdufsha}, and we show it here for the first time for the Navier case. Additionally, we prove that when the inner radius is big enough then the first eigenvalue is always simple. It must be noted though that the threshold we find in this way is not sharp, as can be seen from numerical evidence.


We remark that, although interesting on their own, problems \eqref{laplacian}, \eqref{eq:strongdirev}, and \eqref{eq:strongnavev} are also important in applications as they can be used to model vibrating membranes and plates. Specifically for the latter two, they are derived from the Kirchhoff-Love model of plates (see e.g., \cite{ray}) to study clamped and hinged plates, respectively. As a consequence, shedding light on questions like the ones we consider in this paper is more than mathematical speculation, since it can provide new ideas for the several applications that make use of these problems.


The paper is organized as follows. After some preliminary in Section \ref{sec2}, we compute the eigenfunctions and eigenvalues of radially symmetric domains in Section \ref{sec3}, and then dedicate Section \ref{sec4} to the study of the fundamental tones. In Appendix \ref{appA} we collect some useful facts about Bessel functions.


\section{Preliminaries}\label{sec2}

Let $\Omega \subseteq \R^N$ ($N\ge 2$) be a domain (i.e., an open connected set) with finite Lebesgue measure, and let $H^k(\Omega)$ be the Sobolev space of the (real) functions in $L^2(\Omega)$ whose derivatives up to order $k$ are also in $L^2(\Omega)$, equipped with the standard norm
\begin{equation*}
    \| u \|_{H^k(\Omega)} = \sum_{|\alpha|=0}^{k}\frac{|\alpha|!}{\alpha!} \| \partial^\alpha u \|_{L^2(\Omega)}.
\end{equation*}
Let also $H^k_0(\Omega)$ be the closure in $H^k(\Omega)$ of the set $C^{\infty}_{c}(\Omega)$ of smooth functions with compact support.

The value $\lambda \in \R$ is called an eigenvalue for the Dirichlet bilaplacian if there exists a function $u \in H^2_0(\Omega)$ satisfying
\begin{equation}\label{eq:weakdirev}
    \int_\Omega \Delta u \Delta \varphi \,dx = \lambda \int_{\Omega} u \varphi \, dx \quad \forall \varphi \in H^2_0(\Omega).
\end{equation}
Note that the bilinear form on the left hand side of \eqref{eq:weakdirev} can be associated with a densely defined selfadjoint operator with compact resolvent in $L^2(\Omega)$, implying that there exists a non-decreasing sequence of (strictly positive) eigenvalues of finite multiplicity satisfying equation \eqref{eq:weakdirev}
\begin{equation*}
    0 < \lambda_1^D(\Omega) \leq \lambda_2^D(\Omega) \leq \lambda_3^D(\Omega) \leq \dots
\end{equation*}
and, if $\partial\Omega$ is smooth enough, by standard regularity theory (see \cite[Section 2.5]{ggs}), the eigenfunctions all belong to $H^4(\Omega)$ and therefore satisfy the strong equation \eqref{eq:strongdirev}. In addition, the eigenvalues enjoy the following variational characterization
\begin{equation*}
    \lambda_k^D(\Omega) = \min_{\substack{V_k \subset H^2_0(\Omega)\\ \dim V_k=k}} \max_{u \in V_k\setminus \{0\}}  \frac{\displaystyle\int_{\Omega} (\Delta u)^2 \, dx}{\displaystyle\int_{\Omega} u^2 \, dx}.
\end{equation*}
We remark that, differently from what happens for the eigenvalues of second-order problems, the first eigenvalue $\lambda_1^D(\Omega)$ need not be simple, and the associated eigenfunction may change sign (see \cite{buosokennedy,ggs} and the references therein for additional details).

Now let us set $C_0(\overline{\Omega}) = \{ f \in C(\overline{\Omega}) \colon f_{|\partial\Omega} = 0 \}$, and set $H^2_N(\Omega)$ to be the completion in $H^2(\Omega)$ of $H^2(\Omega) \cap C_0(\overline{\Omega})$. Heuristically, $H^2_N(\Omega)$ is the subset of all the functions that vanish at the boundary. In general, whenever it is possible to define a trace operator $\mathrm{Tr} \colon H^2(\Omega) \to L^2(\partial\Omega)$, $\mathrm{Tr}\,(f)= f_{|\partial\Omega}$, then  
\begin{equation*}
    H^2_N(\Omega) = \{ f \in H^2(\Omega) \colon \mathrm{Tr}\,(f)=0 \}.
\end{equation*}
If the boundary is smooth enough (e.g., Lipschitz continuous), then $H^2_N(\Omega)$ coincides with $H^1_0(\Omega) \cap H^2(\Omega)$, while we will see in Section \ref{sec3} that the punctured ball is an example of a domain where they are different.

An interesting question is to try to set problem \eqref{eq:weakdirev} on $H^2_N(\Omega)$. Unfortunately, the quadratic form in \eqref{eq:weakdirev} is not coercive in general, hence we consider the following problem
\begin{equation}\label{eq:weaknavev}
    \int_\Omega ((1-\sigma) D^2 u : D^2 \varphi + \sigma \Delta u \Delta \varphi) \,dx = \lambda \int_{\Omega} u \varphi \, dx \quad \forall \varphi \in H^2_N(\Omega),
\end{equation}
which is called the Navier eigenvalue problem, where $\sigma$ is a constant called Poisson coefficient and $D^2 u : D^2 \varphi$ denotes the Frobenius product
\begin{equation*}
    D^2 u : D^2 \varphi = \sum_{|\alpha|=2} \frac{2}{\alpha!}\partial^\alpha u \, \partial^\alpha \varphi.
\end{equation*}
This bilinear form is coercive in $H^2_N(\Omega)$ for all $\sigma \in \left(-\frac{1}{N-1}, 1\right)$ due to the inequality
\begin{equation*}
    |D^2 u : D^2 u| = |D^2 u|^2 \geq \frac{1}{N} (\Delta u)^2 \quad \forall u \in H^2(\Omega). 
\end{equation*}
Now problem \eqref{eq:weaknavev} can be associated with a densely defined selfadjoint operator with compact resolvent in $L^2(\Omega)$, implying that there exists a non-decreasing sequence of (strictly positive) eigenvalues of finite multiplicity
satisfying equation \eqref{eq:weaknavev}
\begin{equation*}
    0 < \lambda_1^{N,\sigma}(\Omega) \leq \lambda_2^{N,\sigma}(\Omega) \leq \lambda_3^{N,\sigma}(\Omega) \leq \dots
\end{equation*}
and, if $\partial\Omega$ is smooth enough, by standard regularity theory (see \cite[Section 2.5]{ggs}), the eigenfunctions all belong to $H^4(\Omega)$ and therefore satisfy the strong equation \eqref{eq:strongnavev}. In addition, the eigenvalues enjoy the following variational characterization
\begin{equation*}
    \lambda_k^{N,\sigma}(\Omega) = \min_{\substack{V_k \subset H^2_N(\Omega)\\ \dim V_k=k}} \max_{u \in V_k\setminus \{0\}}  \frac{\displaystyle\int_{\Omega} (1-\sigma) |D^2 u|^2 + \sigma (\Delta u)^2 \, dx}{\displaystyle\int_{\Omega} u^2 \, dx}.
\end{equation*}
If $\Omega$ enjoys additional regularity properties (e.g., if $\partial\Omega$ is convex or of class $C^2$), then problem \eqref{eq:weaknavev} can be properly defined also for $\sigma=1$ and it is possible to show that the eigenvalues $\lambda_k^{N,1}$ are precisely the squares of those of the Dirichlet Laplacian, and the corresponding eigenspaces coincide (see \cite{ggs}).

We observe that, since
\begin{equation*}
    \int_\Omega ((1-\sigma) D^2 u : D^2 \varphi + \sigma \Delta u \Delta \varphi) \,dx = \int_{\Omega} \Delta u \Delta \varphi \, dx 
\end{equation*}
for any $u,\varphi \in H^2_0(\Omega)$, together with the inclusion $H^2_0(\Omega) \subseteq H^2_N(\Omega)$, then there is the natural inequality
\begin{equation}\label{ineq:lambdaN<lambdaD}
    \lambda_k^{N,\sigma}(\Omega) \leq \lambda_k^D(\Omega),
\end{equation}
for all $k \in \N$, for all $\displaystyle\sigma \in \left(-\frac{1}{N-1}, 1\right)$, and for any domain $\Omega$ of finite measure. If in addition $\Omega$ is
smooth, then inequality \eqref{ineq:lambdaN<lambdaD} holds also for $\sigma=1$.

Finally, we observe that, if $\Omega_1\subseteq\Omega_2$, then $\lambda_k^{D}(\Omega_1)\ge \lambda_k^{D}(\Omega_2)$ due to the natural embedding $H^2_0(\Omega_1)\subseteq H^2_0(\Omega_2)$. It is instead not clear whether such an inequality hold for the Navier case: in fact, we will see in Section \ref{sec4} some numerical evidence against it.


\section{Radially symmetric domains}\label{sec3}

In this section we compute somewhat explicitly the eigenvalues and the eigenspaces of problems \eqref{eq:weakdirev} and \eqref{eq:weaknavev} on balls, punctured balls, and spherical shells, centered at the origin. For any $a \in [0,1)$, we set
\begin{equation*}
    B = \{ x \in \R^N \, \colon |x| < 1 \}, \quad B_a = \{ x \in \R^N \, \colon a < |x| < 1 \}.
\end{equation*}
In particular, $B_0$ represent the punctured ball. Following \cite{bf25,chas} we deduce that the eigenfunctions of both \eqref{eq:weakdirev} and \eqref{eq:weaknavev} can be written as a radial function times a spherical harmonic, that is any such eigenfunction $u$ satisfies
\begin{equation}\label{eq:uradial}
    u(x) = f(r) S_\ell(\theta), 
\end{equation}
where $(r,\theta) \in \R_+ \times \Sf^{N-1}$ are the canonical spherical coordinates in $\R^N$, and $S_\ell$ satisfies
\begin{equation*}
    -\Delta_{\Sf^{N-1}} S_\ell(\theta) = \ell(\ell+N-2) S_\ell(\theta),
\end{equation*}
for $\ell \in \N_0$, where $\Delta_{\Sf^{n-1}}$ represents the Laplace-Beltrami operator on $\Sf^{n-1}$. As for the determination of
the radial part, we can rewrite the equation $\Delta^2 u = \lambda u$ as
\begin{equation*}
    (\Delta - t^2)(\Delta + t^2) u = 0,
\end{equation*}
where we have set $t = \lambda^{\frac{1}{4}}$. As the operators $(\Delta - t^2)$ and $(\Delta + t^2)$ commute, by separation of variables we obtain that
\begin{equation}\label{eq:fradial}
    f(r) = c_1 j_{\ell}(tr) + c_2 y_{\ell}(tr) + c_3 i_{\ell}(tr) + c_4 k_{\ell}(tr).
\end{equation}
Here $c_1,\dots,c_4$ are constants that have to be determined, and the functions $j_{\ell}, y_{\ell}, i_{\ell}, k_{\ell}$ are the ultraspherical Bessel functions defined as follows
\begin{equation*}
\begin{aligned}
&j_{\ell} (z) = z^{1-\frac{N}{2}} J_{\ell+\frac{N}{2}-1} (z) && y_{\ell} (z) = z^{1-\frac{N}{2}} Y_{\ell+\frac{N}{2}-1} (z) 
\\
&i_{\ell} (z) = z^{1-\frac{N}{2}} I_{\ell+\frac{N}{2}-1} (z) && k_{\ell} (z) = z^{1-\frac{N}{2}} K_{\ell+\frac{N}{2}-1} (z) 
\end{aligned}
\end{equation*}
We refer to \cite{olvernist} for the definitions and basic properties of Bessel functions (see also Appendix \ref{appA}).

In order to determine the coefficients $c_i$ and the eigenvalue $t^4$, we have to impose the boundary conditions. Starting with the case of the ball $B$, we first observe that the eigenfunctions are of class at least $C^\infty$ (see \cite[Section 2.5]{ggs}), and since no combination of $y_{\ell}$ and $k_{\ell}$ is smooth in the origin, we deduce that $c_2 = c_4 = 0$. On the other hand, the boundary conditions now read
\begin{equation*}
    f(r)_{|r=1} = 0, \quad f'(r)_{|r=1} = 0,
\end{equation*}
for the Dirichlet problem \eqref{eq:strongdirev}, while 
\begin{equation*}
    f(r)_{|r=1} = 0, \quad f''(r) + \sigma \frac{N-1}{r}f'(r)_{|r=1} = 0,
\end{equation*}
for the Navier problem \eqref{eq:strongnavev}. The imposition of the boundary conditions yields a $2\times 2$ linear system in
the unknowns $c_1,c_3$ that must have a nontrivial kernel in order to have nontrivial solutions, and in turn this provides an equation for $t$.
In particular, in the case of the Dirichlet problem \eqref{eq:strongdirev} we are led to the homogeneous system with associated matrix
\begin{equation*}
    M^{D}_1 (t) =
    \begin{pmatrix}
    j_{\ell}(t)
    &i_{\ell}(t)
    \\
    tj'_{\ell}(t)
    &ti'_{\ell}(t)
    \end{pmatrix},
\end{equation*}
and in order to have nontrivial solutions we must impose
\begin{equation*}
    0=\det M^{D}_1(t) = t \Big(  j_{\ell}(t) i'_{\ell}(t) - i_{\ell}(t) j'_{\ell}(t)\Big) = t \Big(  j_{\ell}(t) i_{\ell+1}(t) + j_{\ell+1}(t) i_{\ell}(t) \Big).
\end{equation*}
Here we used the differentiation properties of Bessel functions to simplify the determinant (see Appendix \ref{appA}).

In the case of the Navier problem \eqref{eq:strongnavev} we are instead led to the homogeneous system with associated
matrix
\begin{equation*}
    M^{N}_1(t) =
    \begin{pmatrix}
    j_{\ell}(t)
    &i_{\ell}(t)
    \\
    t^2j''_{\ell}(t) + \sigma (N-1) t  j'_{\ell}(t)
    &t^2i''_{\ell}(t) + \sigma (N-1) t  i'_{\ell}(t)
    \end{pmatrix},
\end{equation*}
and in order to have nontrivial solutions we must impose
\begin{equation}\label{eq:detM^N_1=0}
\begin{aligned}
    0=\det M^{N}_1(t) &= t^2\Big(j_{\ell}(t)i''_{\ell}(t)- j''_{\ell}(t)i_{\ell}(t)\Big) - t \sigma (N-1) \Big(j_{\ell}(t)i'_{\ell}(t)- j'_{\ell}(t)i_{\ell}(t)  \Big)
    \\
    &= 2t^2 j_{\ell}(t)i_{\ell}(t) - t(N-1)(1-\sigma) \Big( j_{\ell}(t) i_{k+1}(t) + j_{k+1}(t) i_{\ell}(t) \Big).
\end{aligned}
\end{equation}
Notice that, for $\sigma=1$, equation \eqref{eq:detM^N_1=0} gives precisely the (square of the) eigenvalues of the Dirichlet Laplacian as solutions.

Turning now to spherical shells $B_a$, for $a \in (0,1)$, the situation becomes a little more involved as we cannot say a priori that any of the coefficients $c_i$ vanish. The boundary conditions now read
\begin{equation*}
    f(r)_{|r=a,1} = 0, \quad f'(r)_{|r=a,1} = 0,
\end{equation*}
for the Dirichlet problem \eqref{eq:strongdirev}, while 
\begin{equation*}
    f(r)_{|r=a,1} = 0, \quad f''(r) + \sigma \frac{N-1}{r}f'(r)_{|r=a,1} = 0,
\end{equation*}
for the Navier problem \eqref{eq:strongnavev}.
The imposition of the boundary conditions yields a $4\times 4$ linear system in
the unknowns $c_1,\dots, c_4$ that must have a nontrivial kernel in order to have nontrivial solutions, and in turn this provides an equation for t. In particular, in the case of the Dirichlet problem \eqref{eq:strongdirev} we are led to the homogeneous system with associated matrix
\begin{equation*}
    M^{D}_{a}(t) =
    \begin{pmatrix}
    j_{\ell}(t)
    &y_{\ell}(t)
    &i_{\ell}(t)
    &k_{\ell}(t)
    \\
    tj'_{\ell}(t)
    &ty'_{\ell}(t)
    &ti'_{\ell}(t)
    &tk'_{\ell}(t)
    \\
    j_{\ell}(ta)
    &y_{\ell}(ta)
    &i_{\ell}(ta)
    &k_{\ell}(ta)
    \\
    tj'_{\ell}(ta)
    &ty'_{\ell}(ta)
    &ti'_{\ell}(ta)
    &tk'_{\ell}(ta)
    \end{pmatrix},
\end{equation*}
and in order to have nontrivial solutions we must impose
\begin{equation}\label{eq:detM^D_a=0}
    0=\det M^{D}_{a}(t) = t^2 \det 
    \begin{pmatrix}
    j_{\ell}(t)
    &y_{\ell}(t)
    &i_{\ell}(t)
    &k_{\ell}(t)
    \\
    -j_{\ell+1}(t)
    &-y_{\ell+1}(t)
    &i_{\ell+1}(t)
    &-k_{\ell+1}(t)
    \\
    j_{\ell}(ta)
    &y_{\ell}(ta)
    &i_{\ell}(ta)
    &k_{\ell}(ta)
    \\
    -j_{\ell+1}(ta)
    &-y_{\ell+1}(ta)
    &i_{\ell+1}(ta)
    &-k_{\ell+1}(ta)
    \end{pmatrix}.
\end{equation}
In the case of the Navier problem \eqref{eq:strongnavev} we are instead led to the homogeneous system with associated
matrix
\begin{equation*}
    M^{N}_{a}(t) =
    \begin{pmatrix}
    (M^{N}_{a})_{1,1}(t)
    &(M^{N}_{a})_{1,2}(t)
    \\
    (M^{N}_{a})_{2,1}(t)
    &(M^{N}_{a})_{2,2}(t)
    \end{pmatrix},
\end{equation*}
where 
\begin{equation*}
\begin{aligned}
    &(M^{N}_{a})_{1,1}(t) = 
    \begin{pmatrix}
    j_{\ell}(t)
    &y_{\ell}(t)
    \\
    t^2j''_{\ell}(t) + t\sigma(N-1) j'_{\ell}(t)
    &t^2y''_{\ell}(t) + t\sigma(N-1)y'_{\ell}(t)
    \end{pmatrix},
    \\
    &(M^{N}_{a})_{1,2}(t) =
    \begin{pmatrix}
    i_{\ell}(t)
    &k_{\ell}(t)
    \\
    t^2i''_{\ell}(t) + t\sigma(N-1)i'_{\ell}(t)
    &t^2k''_{\ell}(t) + t\sigma (N-1) k'_{\ell}(t)
    \end{pmatrix},
    \\
    &(M^{N}_{a})_{2,1}(t) = 
    \begin{pmatrix}
    j_{\ell}(ta)
    &y_{\ell}(ta)
    \\
    t^2j''_{\ell}(ta) + \frac{t}{a}\sigma(N-1) j'_{\ell}(ta)
    &t^2y''_{\ell}(ta) + \frac{t}{a}\sigma(N-1)y'_{\ell}(ta)
    \end{pmatrix},
    \\
    &(M^{N}_{a})_{2,2}(t) = 
    \begin{pmatrix}
    i_{\ell}(ta)
    &k_{\ell}(ta)
    \\
    t^2i''_{\ell}(ta) + \frac{t}{a}\sigma(N-1) i'_{\ell}(ta)
    &t^2k''_{\ell}(ta) + \frac{t}{a}\sigma(N-1)k'_{\ell}(ta)
    \end{pmatrix},
\end{aligned}
\end{equation*}
and in order to have nontrivial solutions we must impose
\begin{equation}\label{eq:detM^N_a=0}
    0=\det M^{N}_{a}(t) = \frac{t^2}{a} \det  \Tilde{M}^{N}_{a}(t) 
\end{equation}
where we have set
\begin{equation*}
    \Tilde{M}^{N}_{a}(t) =
    \begin{pmatrix}
    (\Tilde{M}^{N}_{a})_{1,1}(t)
    &(\Tilde{M}^{N}_{a})_{1,2}(t)
    \\
    (\Tilde{M}^{N}_{a})_{2,1}(t)
    &(\Tilde{M}^{N}_{a})_{2,2}(t)
    \end{pmatrix},
\end{equation*}
for
\begin{equation*}
\begin{aligned}
    &(\Tilde{M}^{N}_{a})_{1,1}(t) = 
    \begin{pmatrix}
    j_{\ell}(t)
    &y_{\ell}(t)
    \\
    (N-1)(1-\sigma) j_{\ell+1}(t)
    &(N-1)(1-\sigma)y_{\ell+1}(t)
    \end{pmatrix},
    \\
    &(\Tilde{M}^{N}_{a})_{1,2}(t) =
    \begin{pmatrix}
    i_{\ell}(t)
    &k_{\ell}(t)
    \\
    -(N-1)(1-\sigma) i_{\ell+1}(t) + 2 t i_{\ell}(t)
    &(N-1)(1-\sigma)k_{\ell+1}(t) + 2t k_{\ell}(t)
    \end{pmatrix},
    \\
    &(\Tilde{M}^{N}_{a})_{2,1}(t) = 
    \begin{pmatrix}
    j_{\ell}(ta)
    &y_{\ell}(ta)
    \\
    (N-1)(1-\sigma) j_{\ell+1}(ta)
    &(N-1)(1-\sigma)y_{\ell+1}(ta)
    \end{pmatrix},
    \\
    &(\Tilde{M}^{N}_{a})_{2,2}(t) = 
    \begin{pmatrix}
    i_{\ell}(ta)
    &k_{\ell}(ta)
    \\
    -(N-1)(1-\sigma) i_{\ell+1}(ta) + 2 at i_{\ell}(ta)
    &(N-1)(1-\sigma)k_{\ell+1}(ta) + 2at k_{\ell}(ta)
    \end{pmatrix}.
\end{aligned}
\end{equation*}
Notice, once more, that for $\sigma = 1$ equation \eqref{eq:detM^N_a=0} gives precisely the (squares of the) eigenvalues of the
Dirichlet Laplacian as solutions. Also, we have not written explicitly the determinants in \eqref{eq:detM^D_a=0} and \eqref{eq:detM^N_a=0} as they are
involved equations that bring no additional information on the eigenvalues.

We finally turn to the case of the punctured ball $B_0$. This case has to be handled in a different way since the boundary of $B_0$ is no longer Lipschitz, and in particular the traces of $H^2$-functions have to be properly interpreted. In particular, when $\Omega= B_0$ then the strong formulations \eqref{eq:strongdirev} and \eqref{eq:strongnavev} do not hold anymore as they are but have to be suitably modified (see also \cite{bfgm} for the discussion of a related bilaplacian problem).

To this end, let us recall that it is possible to define an $H^2$-capacity measure such that, whenever a set $E$ has zero $H^2$-capacity, the following identification holds:
\begin{equation*}
    H^2(\Omega)= H^2(\Omega \setminus E),
\end{equation*}
in the sense that the standard embedding $H^2(\Omega) \hookrightarrow  H^2(\Omega \setminus E)$ is surjective.
We refer to \cite[Section 2.6]{zie}
for the precise definition and basic properties of the capacity applied to Sobolev spaces. In particular, if
$N\geq 4$, a singleton has vanishing $H^2$-capacity (see \cite[Theorem 2.6.16]{zie}), so that
\begin{equation*}
    H^2(B_0) = H^2(B).
\end{equation*}
In addition, since it is not possible to properly define any trace of an $H^2$-function on a singleton for $N\geq 4$, this means that
\begin{equation*}
    H^2_0(B_0) = H^2_0(B).
\end{equation*}
hence the weak formulations \eqref{eq:weakdirev} and \eqref{eq:weaknavev} set in $B_0$ are equivalent to those set in $B$, producing the same solutions (respectively). In particular, the hole in $B_0$ is not seen by the problem, and therefore the
strong formulations coincide with those in $B$.

On the other hand, if $N = 2, 3$ a singleton has now a non-vanishing $H^2$-capacity (see \cite[Remark 2.6.15]{zie}), meaning that it is now possible for the embedding $H^2(\Omega) \hookrightarrow  H^2(\Omega \setminus E)$ not to be surjective. Hence
we must analyze the behaviour of any $u \in H^2(B_0)$ at the origin.

We observe that $\nabla u \in (H^1(B_0))^N = (H^1(B))^N$, since singletons have zero $H^1$-capacity for any $N\geq 2$. In
particular, $\nabla u$ is an $L^2$-vector and is continuous on almost every straight line passing through the origin.
Moreover, it is not possible to trace the gradient on a singleton for any $N \geq 2$ (see also \cite[Section 5.2]{bur}).
On the other hand, by the Sobolev Embedding Theorem (cf. \cite[Section 4.6]{bur}) we know that $u \in C_b(B_0)$,
meaning in particular that $u$ has to be bounded around the origin. Since it is also possible to trace $u$
on a singleton if $N = 2, 3$, we infer that $u$ can be continuously extended at the origin, and therefore
$H^2(B_0) = H^2(B)$.

Let us then consider $u \in H^2_0(B_0)$. Since $u$ can be traced at the origin, we infer that $u(0) = 0$. However,
$\nabla u$ cannot be traced at the origin, and in particular the expression $\frac{\partial u}{\partial \nu}$ does not make sense at the origin, while it still holds that $\nabla u_{|\partial B} = 0$. The same considerations allow to conclude that, if $u \in H^2_{N}(B_0)$, then $u(0) = 0$.

In particular, the Dirichlet problem on $B_0$ for $N = 2, 3$ reads

\begin{equation}\label{eq:strongdirev B0}
\begin{cases}
    \Delta^2 u = \lambda u, & \text{in } B_0,
    \\
    u=\partial_\nu u = 0, & \text{on } \partial B,
    \\
    u(0) = 0,
\end{cases}
\end{equation}
while the Navier problem reads
\begin{equation}\label{eq:strongnavev B0}
\begin{cases}
    \Delta^2 u = \lambda u, & \text{in }  B_0,
    \\
    u=(1+\sigma)\partial_\nu^2 u + \sigma \Delta u = 0, & \text{on } \partial B,
    \\
    u(0) = 0.
\end{cases}
\end{equation}
Now for both the Dirichlet and the Navier problem on $B_0$, if $N \geq 4$ then the eigenvalues and eigenfunctions coincide with those of the ball $B$. If instead $N = 2, 3$, the eigenfunctions can still be written as in \eqref{eq:uradial} with $f$ in the form \eqref{eq:fradial}, however the boundary condition $u(0) = 0$ forces the solution on the origin, providing in turn some conditions on the coefficients $c_i$.
We observe that the functions $y_{\ell}$ and $k_{\ell}$ are not bounded at the origin, and since $u$ is instead continuous
with an $L^2$ gradient, the coefficients $c_2$ and $c_4$ are bound to be such that the combination
\begin{equation*}
    c_2 y_{\ell}(tr) + c_4 k_{\ell}(tr) 
\end{equation*}
reflect this regularity at the origin (note that $j_{\ell}, i_{\ell}$ are of class $C^\infty$ at the origin).

Let us start from $N = 2$, where we are led to consider the combination
\begin{equation*}
    c_2 Y_{\ell}(z) + c_4 K_{\ell}(z).
\end{equation*}
From the series development of $Y_{\ell}, K_{\ell}$ (see Appendix \ref{appA}) it is clear that the only admissible coefficients for $\ell \geq 3$ are $c_2 =c_4 = 0$. For $\ell = 2$ we see that
\begin{equation*}
    c_2 Y_2(z) + c_4 K_2(z) = - c_2 \frac{4}{\pi z^2} \left( 1+\frac{z^2}{4}\right) +  c_4 \frac{2}{z^2} \left( 1-\frac{z^2}{4}\right) + P_1(z)
\end{equation*}
where $P_1$ is smooth. This implies that $c_4 = \frac{2}{\pi} c_2$. However, since
\begin{equation*}
    c_2 Y'_2(z) + \frac{2}{\pi} c_2 K'_2(z) = c_2 \left( J_1(z) + I_1(z) \right) \ln \sqrt{\frac{z}{2}} + P_2(z),
\end{equation*}
where $P_2$ is smooth, recalling that $J_1(0) + I_1(0) =1$, we conclude that the only admissible coefficients
for $\ell = 2$ are $c_2 = c_4 = 0$. 

For $\ell = 1$ we have
\begin{equation*}
    c_2 Y_1(z) + c_4 K_1(z) = - c_2 \frac{2}{\pi z} +  c_4 \frac{2}{z} \left( 1-\frac{z^2}{4}\right) + P_3(z),
\end{equation*}
where $P_3$ is smooth. This implies again that $c_4 = \frac{2}{\pi} c_2$. However, since
\begin{equation*}
    c_2 Y'_2(z) + \frac{2}{\pi} c_2 K'_2(z) = c_2 \frac{J_0(z) + I_0(z)}{2z} \ln \frac{z}{2} + P_4(z),
\end{equation*}
where $P_4$ is smooth, recalling that $J_0(0) + I_0(0) = 2$,
we conclude that the only admissible coefficients for
$\ell = 1$ are $c_2 = c_4 = 0$ as well.

Finally, for $\ell = 0$ we have
\begin{equation*}
    c_2 Y_0(z) + c_4 K_0(z) = c_2 \frac{2}{\pi} \left(\ln \frac{z}{2} + \gamma_{EM} \right)J_0(z) -  c_4 \left(\ln \frac{z}{2} + \gamma_{EM} \right)I_0(z) + P_5(z),
\end{equation*}
where $P_5$ is smooth and $\gamma_{EM}$ is the Euler-Mascheroni constant (see Appendix \ref{appA}). This implies again that $c_4 = \frac{2}{\pi} c_2$, and it is now easy to check that the function
\begin{equation*}
    Y_0(z) + \frac{2}{\pi} K_0(z)
\end{equation*}
is of class $C^2$ at the origin.

As for dimension $N = 3$, the respective expansions (see Appendix \ref{appA}) immediately shows that $c_2 = c_4 = 0$ for $\ell \geq 1$, while $c_4 = \frac{\pi}{2} c_2$ for $\ell = 0$.

Summing up all these arguments, we obtain that the eigenfunctions and eigenvalues of the punctured ball $B_0$ (for $N = 2, 3$) coincide with those of the ball $B$ for all $\ell \geq 1$, while they differ for $\ell = 0$, i.e., the only eigenvalues which are different are those associated with radial eigenfuctions (since the zero-order spherical harmonic $S_0(\theta)$ are constants), which take the form
\begin{equation*}
    u(x) = d_1 \left( j_0(t|x|) - i_0(t|x|) \right) + d_2 \left( y_0(t|x|) + \frac{2}{\pi} k_0(t|x|) \right).
\end{equation*}
In the Dirichlet case \eqref{eq:strongdirev B0}, imposing the boundary conditions on $\partial B$ leads us to a homogeneous linear system with associated matrix
\begin{equation*}
    M^{D}_{0}(t) =
    \begin{pmatrix}
    j_0(t) - i_0(t)
    &y_0(t) + \frac{2}{\pi}k_0(t)
    \\
    tj'_0(t) - ti'_0(t)
    &ty'_0(t) + \frac{2t}{\pi}k'_0(t)
    \end{pmatrix},
\end{equation*}
and in order to have nontrivial solutions we must impose
\begin{equation*}
    0= \det  M^{D}_{0}(t) = t \left(j_1(t) + i_1(t) \right) \left(y_0(t) + \frac{2}{\pi}k_0(t) \right) - t \left(j_0(t) - i_0(t) \right) \left(y_1(t) + \frac{2}{\pi}k_1(t) \right).
\end{equation*}
We observe that this discussion generalizes the results already obtained in \cite{cofdufsha}.

In the Navier case \eqref{eq:strongnavev B0}, imposing the boundary conditions on $\partial B$ leads us to a homogeneous linear
system with associated matrix
\begin{multline*}
    M^{N}_{0}(t) =\\
    \begin{pmatrix}
    j_0(t) - i_0(t)
    &y_0(t) + \frac{2}{\pi}k_0(t)
    \\
    t^2(j''_0(t) - i''_0(t)) + t\sigma(N-1) (j'_0(t) - i'_0(t))
    &t^2\left(y''_0(t) + \frac{2}{\pi}k''_0(t)\right) + t \sigma (N-1)\left(y'_0(t) + \frac{2}{\pi}k'_0(t)\right)
    \end{pmatrix},
\end{multline*}
and in order to have nontrivial solutions we must impose
\begin{equation*}
\begin{aligned}
    0&= \det  M^{N}_{0}(t) 
    \\
    &= t \det 
    \begin{pmatrix}
    j_0(t) - i_0(t)
    &y_0(t) + \frac{2}{\pi}k_0(t)
    \\
    (N-1)(1-\sigma) (j_1(t) + i_1(t)) - 2t i_0(t)
    &(N-1)(1-\sigma) \left(y_1(t) + \frac{2}{\pi}k_1(t)\right) + \frac{4}{\pi} t k_0(t)
    \end{pmatrix}.
\end{aligned}
\end{equation*}
We remark how in this case, posing $\sigma=1$ does not provide the (squares of the) zeros of the Bessel function as solutions at all. This reflect the fact that problem \eqref{eq:weaknavev} is not well defined in the punctured ball as the associated quadratic form is not coercive.

\section{Fundamental tones}\label{sec4}

In this section we analize in more detail the behavior of the eigenfunctions in order to identify the
fundamental tones of balls, punctured balls, and annuli/spherical shells. In particular, we recall from
Section \ref{sec3} that any eigenfunction can be written as
\begin{equation*}
    u(x) = f_\ell(r) S_\ell(\theta),
\end{equation*}
and we are interested in identifying at which phase order $\ell$ the fundamental tone occurs. We follow the
ideas of \cite{cofdufsha} of making explicit the dependence of the Rayleigh quotient upon $\ell$.
Let us start by considering the identity
\begin{equation*}
    \int_\Omega (1-\sigma) |D^2 v|^2 + \sigma (\Delta v)^2 \,dx = \int_{\Omega} v \Delta^2 v \, dx + \int_{\partial^*\Omega} \partial_\nu v ((1-\sigma)\partial^2_\nu v + \sigma \Delta v) \, d \mathcal{H}^{N-1},
\end{equation*}
valid for any $v \in H^2_N(\Omega)$, where $\Omega$ is either $B$ or $B_a$ for some $a \in [0, 1)$, and $\partial^*\Omega = \partial\Omega$ unless $\Omega= B_0$ in which case  $\partial^*\Omega = \partial B$. Let us also recall that the Laplacian in spherical coordinates becomes
\begin{equation*}
    \Delta v = \Delta_r v + \frac{1}{r^2} \Delta_{\mathbb{S}^{N-1}} v,
\end{equation*}
where we have set
\begin{equation*}
    \Delta_r v = r^{1-N} \partial_r \left(r^{N-1} \partial_r \right) v = \partial^2_r v + \frac{N-1}{r} \partial_r v.
\end{equation*}
This implies that 
\begin{equation*}
    \Delta^2 u = \Delta^2 (f_{\ell} S_{\ell}) = \left( \Delta^2_r f_{\ell} - \mu_{\ell} \Delta_r \frac{f_{\ell}}{r^2} - \frac{\mu_{\ell}}{r^2} \Delta_r f_{\ell} + \frac{\mu_{\ell}}{r^4} f_{\ell} \right) S_{\ell},
\end{equation*}
where $\mu_{\ell}= \ell(\ell+N-2)$. Recalling that $\displaystyle\int_{\mathbb{S}^{N-1}} S_{\ell}^2 \, d \mathcal{H}^{N-1} = 1$, we have
\begin{equation}\label{eq:weaknavevrad}
\begin{split}
    \int_\Omega (1-\sigma) |D^2 u|^2 + \sigma (\Delta u)^2 \,dx &= \int_{\Omega} u \Delta^2 u \, dx 
    \\
    &= \int_{a}^{1} \left( \Delta^2_r f_{\ell} - \mu_{\ell} \Delta_r \frac{f_{\ell}}{r^2} - \frac{\mu_{\ell}}{r^2} \Delta_r f_{\ell} + \frac{\mu_{\ell}}{r^4} f_{\ell} \right) f_{\ell} r^{N-1} \,dr,
\end{split}
\end{equation}
where if $\Omega=B$ we must take $a=0$ in the right had side. 

We now proceed to compute the various
terms coming from the last integral in \eqref{eq:weaknavevrad}. We have
\begin{equation*}
\begin{aligned}
    \int_{a}^{1} f_{\ell} \Delta^2_r f_{\ell} r^{N-1} \,dr = \int_{a}^{1} f_{\ell} \partial_r(r^{N-1} \partial_r \Delta_r f_{\ell}) \,dr &= - \int_{a}^{1} \partial_r f_{\ell} \partial_r \Delta_r f_{\ell} r^{N-1}  \,dr
    \\
    &=-\left[\partial_r f_{\ell} \Delta_r f_{\ell} r^{N-1}\right] \bigg|_{r=a}^{r=1} + \int_{a}^{1}  (\Delta_r f_{\ell})^2 r^{N-1} \,dr,
\end{aligned}
\end{equation*}
and similarly
\begin{equation*}
\begin{split}
&\int_{a}^{1} f_{\ell} \Delta^2_r f_{\ell} r^{N-1} \, dr = \int_{a}^{1} f_{\ell} \partial_r \left(r^{N-1} \partial_r \Delta_r f_{\ell} \right) \, dr 
\\
&= - \int_{a}^{1} \partial_r f_{\ell} \left( \partial^3_r f_{\ell} + \frac{N-1}{r} \partial^2_r f_{\ell} -  \frac{N-1}{r^2} \partial_r f_{\ell} \right) r^{N-1} \, dr
\\
&= - \int_{a}^{1} \partial_r f_{\ell} \partial^3_r f_{\ell} r^{N-1} \, dr - (N-1) \int_{a}^{1} \partial_r f_{\ell} \partial^2_r f_{\ell} r^{N-2} \, dr + (N-1) \int_{a}^{1} \left(\partial_r f_{\ell} \right)^2 r^{N-3} \, dr
\\
& = \left[\partial_r f_{\ell} \partial^2_{r} f_{\ell} r^{N-1}\right] \bigg|_{r=a}^{r=1} + \int_{a}^{1} (\partial^2_{r} f_{\ell})^2 r^{N-1} \, dr + (N-1) \int_{a}^{1} \left(\partial_r f_{\ell}\right)^2 r^{N-3} \, dr,
\end{split}
\end{equation*}
so that
\begin{equation*}
\int_{a}^{1} f_{\ell} \Delta^2_r f_{\ell} r^{N-1} \, dr = \int_{a}^{1} \Big( (1-\sigma)(\partial^2_{r} f_{\ell})^2 + \sigma (\Delta^2_{r} f_{\ell})^2 \Big) r^{N-1} \, dr + (1-\sigma)(N-1) \int_{a}^{1} (\partial_r f_{\ell})^2 r^{N-3} \, dr,
\end{equation*}
where we used the fact that
\begin{equation*}
\partial_r f_{\ell} \Big((1-\sigma) \partial^2_{r} f_{\ell} + \sigma \Delta_r f_{\ell} \Big) r^{N-1} \bigg|_{r=a,1} = 0.
\end{equation*}
Now
\begin{equation*}
\begin{split}
    -\int_{a}^{1} f_{\ell} \Delta_r \left(\frac{f_{\ell}}{r^{2}}\right) r^{N-1} \, dr &= - \int_{a}^{1} f_{\ell} \partial_r \left(r^{N-1} \partial_r \left(\frac{f_{\ell}}{r^{2}}\right) \right) \, dr 
    \\ 
    &= \left[2f_{\ell}^2 r^{N-4}-f_\ell\partial_r f_\ell r^{N-3}\right] \bigg|^{r=1}_{r=a}+ \int_{a}^{1} (\partial_r f_{\ell})^2 r^{N-3} \, dr -  2\int_{a}^{1} f_{\ell} \partial_r f_{\ell} r^{N-4} \, dr
    \\
    &= \int_{a}^{1} (\partial_r f_{\ell})^2 r^{N-3} \, dr - \left[f_{\ell}^2 r^{N-4}\right] \bigg|^{r=1}_{r=a}+(N-4)	\int_{a}^{1} f_{\ell}^2 r^{N-5} \, dr
    \\
    &= \int_{a}^{1} (\partial_r f_{\ell})^2 r^{N-3} \, dr + (N-4) \int_{a}^{1} f_{\ell}^2 r^{N-5} \, dr,
\end{split}
\end{equation*}
where we used the fact that 
$$f_{\ell}^2 r^{N-4} \bigg|_{r=a,1} =f_\ell\partial_r f_\ell r^{N-3}\bigg|_{r=a,1}= 0$$ 
for $\ell \geq 1$. Note that for $\ell = 0$ this integral is multiplied by $\mu_0 = 0$ in \eqref{eq:weaknavevrad}.

Similarly
\begin{equation*}
\begin{split}
    -\int_{a}^{1} \frac{f_{\ell}}{r^2} &\Delta_r f_{\ell}  r^{N-1} \, dr = - \int_{a}^{1} f_{\ell} \left(\partial^2_r f_\ell+\frac{N-1}{r}\partial_r f_\ell\right) r^{N-3} \, dr 
    \\
    &= - \left[f_{\ell} \partial_r f_{\ell} r^{N-3}\right] \bigg|_{r=a}^{r=1} + \int_{a}^{1} (\partial_r f_{\ell})^2 r^{N-3} \, dr - 2 \int_{a}^{1} f_{\ell} \partial_r f_{\ell}  r^{N-4} \, dr
    \\
    &= - \left[f_{\ell} \partial_r f_{\ell} r^{N-3} + f_{\ell}^2 r^{N-4} \right] \bigg|_{r=a}^{r=1} + \int_{a}^{1} (\partial_r f_{\ell})^2 r^{N-3} \, dr + (N-4) \int_{a}^{1} f_{\ell}^2 r^{N-5} \, dr
    \\
    &= \int_{a}^{1} (\partial_r f_{\ell})^2 r^{N-3} \, dr + (N-4) \int_{a}^{1} f_{\ell}^2 r^{N-5} \, dr,
\end{split}
\end{equation*}
for $\ell \geq 1$, while for $\ell = 0$ this integral is multiplied by $\mu_0 = 0$ in \eqref{eq:weaknavevrad}.

Summing up we get that
\begin{equation*}
\begin{split}
\int_{\Omega} (1-\sigma) |D^2 u|^2 + \sigma (\Delta u)^2 \, dx &= \,\int_{a}^{1} \left((1-\sigma)(\partial^2_{r} f_{\ell})^2 + \sigma (\Delta_r f_{\ell})^2 \right) r^{N-1} \, dr 
\\
&
\quad+ A(\ell,N) \int_{a}^{1} (\partial_r f_{\ell})^2 r^{N-3} \, dr + B(\ell,N) \int_{a}^{1} f_{\ell}^2 r^{N-5} \, dr
\\
&
=: \mathcal{F}(\ell,N,f_{\ell}),
\end{split}
\end{equation*}
where $A(\ell,N) = (2\mu_{\ell} + (1-\sigma)(N-1))$ and $B(\ell,N) = \mu_{\ell} (\mu_{\ell} + 2N-8)$.

For the same reason
\begin{equation*}
\int_{\Omega} |u|^2 \, dx = \int_{a}^{1} f_{\ell}^2 r^{N-1} \, dr,
\end{equation*}
and this tells us that
\begin{equation}
\label{evkph}
\lambda^{\bullet,\sigma}_{1,\ell} (\Omega) = \min_{f \in V^{\bullet} (\Omega)} \dfrac{\mathcal{F}(\ell,N,f)}{\displaystyle\int_{a}^{1} f^2 r^{N-1} \, dr}
\end{equation}
is an eigenvalue, more specifically is the smaller eigenvalue with phase order $\ell$. Here $V^\bullet(\Omega)$ is defined as follows:
\begin{itemize}
    \item $V_D(B_a) = \{f \in C^2(a,1) : f(a) = f(1) = f'(a) = f'(1) = 0\}$ for $a \in (0,1)$;
    \item $V_D(B_0) = \{f \in C^2(0,1) : f(0) = f(1) = f'(1) = 0\}$;
    \item $V_D(B) = \{f \in C^2[0,1) : f(1) = f'(1) = 0\}$;
    \item $V_N(B_a) = \{f \in C^2(a,1) : f(a) = f(1) = 0\}$ for $a \in (0,1)$;
    \item $V_N(B_0) = \{f \in C^2(0,1) : f(0) = f(1) = 0\}$;
    \item $V_N(B) = \{f \in C^2[0,1) : f(1) = 0\}$.
\end{itemize}

Hence, the fundamental tone of the problem will be
\begin{equation*}
\lambda^{\bullet,\sigma}_{1}(\Omega) = \min_{\ell \in \mathbb{N}_0} \lambda^{\bullet,\sigma}_{1,\ell}(\Omega).
\end{equation*}

We now analyze the behavior of the Rayleigh quotient with respect to $\ell$. We observe that the dependence upon $\ell$ is expressed through two coefficients depending on $\ell$. The first one is
\begin{equation*}
A(\ell,N) = 2\ell(\ell+ N-2) + (1-\sigma)(N-1),
\end{equation*}
which is strictly increasing in $\ell \in \mathbb{N}_0$ for any value of the dimension $N$. The second one
\begin{equation*}
B(\ell,N) = \ell(\ell+ N-2)(\ell^2 + \ell(N-2) + 2(N-4))
\end{equation*}
is strictly increasing in $\ell \in \mathbb{N}_0$ only for $N \geq 4$. For $N = 2$ it becomes
\begin{equation*}
B(\ell,2) = \ell^2(\ell^2 -4)
\end{equation*}
which is strictly increasing only for $\ell \geq 1$ with $B(0,2) = 0 > -3 = B(1,2)$, while for $N = 3$
\begin{equation*}
B(\ell,3) = \ell(\ell+ 1)(\ell^2 + \ell -2)
\end{equation*}
is again strictly increasing only for $\ell \geq 1$, and $B(0,3) = B(1,3) = 0$. In particular, all these considerations prove the following

\begin{theorem}\label{ftthm}
For $N \geq 3$, the fundamental tone of the ball $B$, of a spherical shell or of the punctured ball $B_a$ for $a \in [0,1)$, is always simple and associated with a radial eigenfunction, both for the Dirichlet problem \eqref{eq:weakdirev} and for the Navier problem \eqref{eq:weaknavev} (for any $\sigma$ such that the problem is well defined).

For $N = 2$, the fundamental tone of the ball $B$, of an annulus or of the punctured ball $B_a$ for $a \in [0,1)$, can be associated with either a radial eigenfunction or one with angular phase of order 1, both for the Dirichlet problem \eqref{eq:weakdirev} and for the Navier problem \eqref{eq:weaknavev} (for any $\sigma$ such that the problem is well defined). In particular, the fundamental tone can be either simple (in the first case), double (in the second case), or triple (in the bifurcation case).
\end{theorem}

We recall that the result of Theorem \ref{ftthm} is a generalization of \cite{cofdufsha} which analyzes only the Dirichlet problem for $N = 2$. Here we highlight the fact that if the dimension is $N \geq 3$ then the situation is somehow simpler and the fundamental tone of any radial domain is always simple. A fact that seems still out of reach is to prove that in any of those cases the associated eigenfunction is always of one sign.

We recall that in \cite{cofdufsha} it was shown that the fundamental tone of the Dirichlet problem \eqref{eq:weakdirev} on a radial planar domain can be only associated with $\ell = 0$ or $\ell = 1$, and provided numerical computations showing that there exists a critical value $a_0 \approx 0.0013117174$ such that the fundamental tone is radial for $a > a_0$, while it has angular phase $\ell = 1$ for $a < a_0$, meaning in addition that the first eigenvalue is double for $a < a_0$. Moreover, for $a = a_0$ the first eigenvalue is triple. While an analytic proof of this fact seems out of reach at the moment, nevertheless we have the following

\begin{theorem}
\label{a1}
Let $N = 2$, and let
\begin{equation*}
a_1 = e^{-\pi \sqrt{2}} \approx 0.01176198.
\end{equation*}
If $a > a_1$, then the fundamental tone of $B_a$ is simple and the associated eigenfunction is radial. If $a = a_1$ the fundamental tone admits a radial eigenfunction.
\end{theorem}

Note that the value $a_1$ obtained in Theorem \ref{a1} is way worse than the value $a_0$ obtained in \cite{cofdufsha}. Also when considering the Navier problem with any value of $\sigma$ the value $a_1$ is far from optimal: as can be seen from Figure \ref{figura1} the threshold ratio seems to be always below $0.006$. An interesting question would be to understand the dependence of the threshold radius upon $\sigma$.

\begin{figure}[htb]
\centering
\begin{tabular}{cc}

\includegraphics[width=0.40\textwidth]{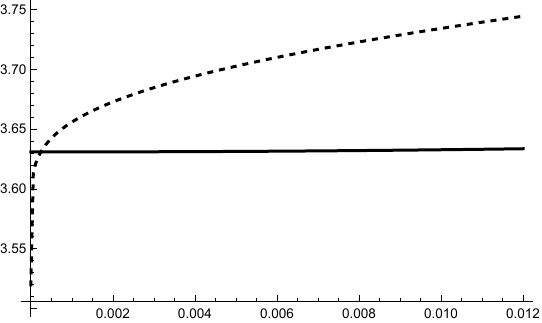} \quad & \quad\includegraphics[width=0.40\textwidth]{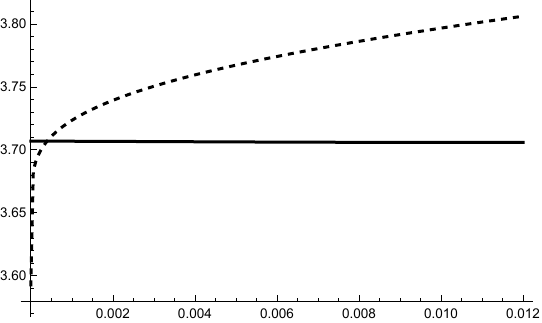} \\
$\sigma=-0.7$ & $\sigma=-0.4$ \\

\includegraphics[width=0.40\textwidth]{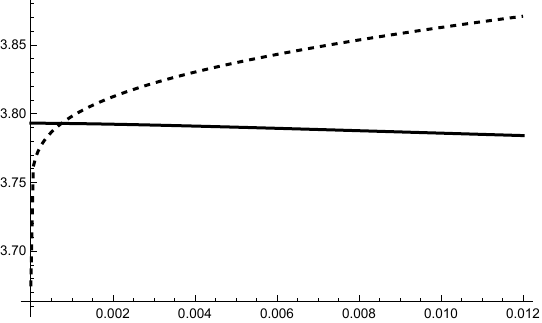} \quad&\quad \includegraphics[width=0.40\textwidth]{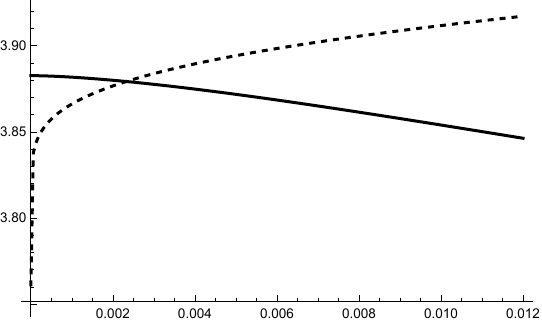} \\

$\sigma=0$ & $\sigma=0.5$ \\

\includegraphics[width=0.40\textwidth]{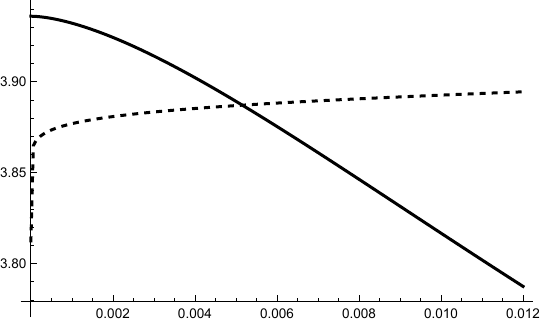} \quad& \quad\includegraphics[width=0.40\textwidth]{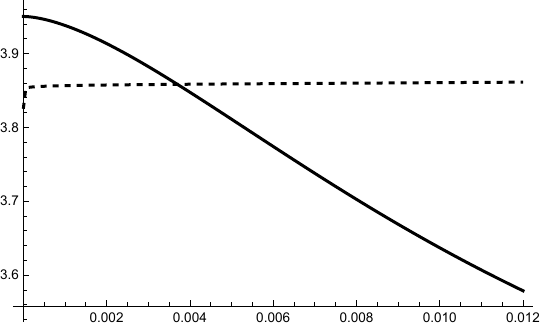} \\

$\sigma=0.85$ & $\sigma=0.95$ \\

\end{tabular}
\caption{The smallest eigenvalue associated with $\ell=0$ (bold) and $\ell=1$ (dashed) for $a\in[0,0.012]$ and for different values of $\sigma$.
}

\label{figura1}
\end{figure}

We observe in passing that Figure \ref{figura1} shows that, differently from the Dirichlet case, in the Navier case we have examples of domains for which $\Omega_1\subseteq \Omega_2$ and $\lambda_1^{N,\sigma}(\Omega_1)\le\lambda_1^{N,\sigma}(\Omega_1)$.

\begin{proof}
We now proceed with estimating for what values of $a$ the quotient in \eqref{evkph} for $\ell = 1$ is bigger than the quotient for $\ell = 0$ for any given $f \in V^\bullet(B_a)$. We first observe that
\begin{equation*}
\begin{split}
\left( \int_{a}^{1} \Big((1-\sigma)(\partial^2_{r} f)^2 + \sigma (\Delta_r f)^2 \Big)\, rdr + A(1,2) \int_{a}^{1} (\partial_r f)^2 r^{-1} \, dr + B(1,2) \int_{a}^{1} f^2 r^{-3} \, dr \right)
\\
- \left( \int_{a}^{1} \Big( (1-\sigma)(\partial^2_{r} f)^2 + \sigma (\Delta_r f)^2 \Big) \, rdr + A(0,2) \int_{a}^{1} (\partial_r f)^2 r^{-1} \, dr + B(0,2) \int_{a}^{1} f^2 r^{-3} \, dr \right)
\\
= 2 \int_{a}^{1} (\partial_r f)^2 r^{-1} \, dr - 3 \int_{a}^{1} f^2 r^{-3} \, dr.
\end{split}
\end{equation*}

The conclusion will follow once we show that the smallest eigenvalue $\lambda_1(a)$ of the problem
\begin{equation}\label{eulereq}
\int_{a}^{1} (\partial_r f)^2 r^{-1} \, dr = \lambda \int_{a}^{1} f^2 r^{-3} \, dr
\end{equation}
is greater than $\frac{3}{2}$ for $a > a_1$ and that $\lambda_1(a_1) = \frac{3}{2}$. From now on we will consider $a > 0$.

We observe that problem \eqref{eulereq} can be associated with a selfadjoint operator with compact resolvent and therefore its spectrum consists of eigenvalues of finite multiplicity, the smallest of which satisfies
\begin{equation*}
\lambda_1(a) = \inf_{f \in V^\bullet(B_a)} \frac{\displaystyle \int_{a}^{1} (\partial_r f)^2 r^{-1} \, dr}{\displaystyle \int_{a}^{1} f^2 r^{-3} \, dr} = \min_{f \in H^1_0(a,1)} \dfrac{\displaystyle \int_{a}^{1} (\partial_r f)^2 r^{-1} \, dr}{\displaystyle \int_{a}^{1} f^2 r^{-3} \, dr}.
\end{equation*}

In particular, equation \eqref{eulereq} can be rewritten in strong form as
\begin{equation}\label{eulereqstrong}
\begin{cases}
    r^2 \partial^2_{r} f - r \partial_r f + \lambda f = 0, \quad r \in (a,1),
    \\
    f(a) = f(1) = 0.
\end{cases}
\end{equation}

Problem \eqref{eulereqstrong} is known as the linear Euler equation and there are several ways to solve it (see e.g., \cite{ince}). One possibility is to operate a substitution of the type $x = e^t$, which leads to the following characteristic equation
\begin{equation}\label{chareu}
m^2 - 2m + \lambda = 0.
\end{equation}
We distinguish three possible cases depending on $\lambda$.

\textbf{Case $0 < \lambda < 1$.} In this case equation \eqref{chareu} admits two real roots
\begin{equation*}
m_\pm = 1 \pm \sqrt{1 - \lambda},
\end{equation*}
hence problem \eqref{eulereqstrong} has solutions of the type
\begin{equation*}
f(r) = \alpha r^{m_-} + \beta r^{m_+},
\end{equation*}
where $\alpha, \beta \in \mathbb{R}$ have to be deduced imposing the boundary conditions. A quick check shows that there are no solutions to problem \eqref{eulereqstrong} in this form, so we conclude that $\lambda \notin (0,1)$.

\textbf{Case $\lambda = 1$.} In this case equation \eqref{chareu} admits one real double root
\begin{equation*}
m = 1,
\end{equation*}
hence problem \eqref{eulereqstrong} has solutions of the type
\begin{equation*}
f(r) = \alpha r^m \ln r + \beta r^m,
\end{equation*}
where $\alpha, \beta \in \mathbb{R}$ have to be deduced imposing the boundary conditions. A quick check shows that there are no solutions to problem \eqref{eulereqstrong} in this form, so we conclude that $\lambda \notin (0,1]$.

\textbf{Case $\lambda > 1$.} In this case equation \eqref{chareu} admits the complex roots
\begin{equation*}
m_\pm = 1 \pm i \sqrt{\lambda - 1},
\end{equation*}
hence problem \eqref{eulereqstrong} has solutions of the type
\begin{equation*}
f(r) = \alpha r \cos\left(\sqrt{\lambda - 1} \ln r \right) + \beta r \sin \left(\sqrt{\lambda - 1} \ln r\right),
\end{equation*}
where $\alpha, \beta \in \mathbb{R}$ have to be deduced imposing the boundary conditions. The condition $f(1) = 0$ readily implies that $\alpha = 0$, while the condition $f(a) = 0$ leads to
\begin{equation*}
a \sin \left(\sqrt{\lambda - 1} \ln a \right) = 0,
\end{equation*}
which in turn provides the following values for the eigenvalues of problem \eqref{eulereqstrong}
\begin{equation*}
\lambda = 1 + \left(\frac{h \pi}{\ln a}\right)^2, \quad \text{for } h \in \mathbb{N} \setminus \{0\}.
\end{equation*}
In particular, the first eigenvalue is
\begin{equation*}
\lambda_1(a) = 1 + \left(\frac{\pi}{\ln a}\right)^2,
\end{equation*}
and
\begin{equation*}
\lambda_1(a) > \frac{3}{2} \quad \text{if and only if} \quad a > e^{-\pi \sqrt{2}}.
\end{equation*}
\end{proof}

\appendix
\section{Useful properties of Bessel functions}\label{appA}
In this appendix we recall some properties of Bessel functions that we used in previous sections; we refer to \cite{olvernist} for general references.

We recall that the Bessel functions of the first kind $J_p$ and of the second kind $Y_p$ are the solutions to Bessel’s equation
\begin{equation*}
z^2 \frac{\partial^2 w}{\partial z^2}(z) + z \frac{\partial w}{\partial z}(z) + (z^2 - p^2) w(z) = 0.
\end{equation*}

Bessel functions admit the following series representations for real arguments (here $\Gamma$ denotes Euler’s Gamma function):
\begin{equation*}
\begin{split}
J_p(x) &= \frac{x^p}{2^p} \sum_{k=0}^{\infty} \frac{1}{k! \Gamma(p + k + 1)} \left(-\frac{x^2}{4}\right)^k,
\\
Y_p(x) &= \frac{J_p(x) \cos(p \pi) - J_{-p}(x)}{\sin(p \pi)}, \quad \text{if } p \notin \mathbb{Z},
\\
Y_p(x) &= \frac{1}{\pi} \frac{\partial J_q(x)}{\partial q} \bigg|_{q=p} + \frac{(-1)^p}{\pi} \frac{\partial J_q(x)}{\partial q} \bigg|_{q=-p}, \quad \text{if } p \in \mathbb{Z}.
\end{split}
\end{equation*}
In particular,
\begin{equation*}
Y_p(x) = -\frac{2^p}{\pi x^p} \sum_{k=0}^{p-1} \frac{(p-k-1)! x^{2k}}{k! 4^k} + \frac{2}{\pi} \ln\left(\frac{x}{2}\right) J_p(x) + Q_p(x), \quad \text{if } p \in \N \setminus \{0\}
\end{equation*}
and
\begin{equation*}
Y_0(x) = \frac{2}{\pi} \left( \gamma_{\text{EM}} + \ln \left(\frac{x}{2}\right) \right) J_0(x) + Q_0(x),
\end{equation*}
where $\gamma_{\text{EM}}$ is the Euler-Mascheroni constant and $Q_p$ are $C^\infty$-functions (see \cite[Sec. 10.8]{olvernist} for more details).

The ultraspherical Bessel functions $j_p, y_p$ are defined as the solutions to the ultraspherical Bessel equation
\begin{equation*}
z^2 \frac{\partial^2 w}{\partial z^2}(z) + z(N-1) \frac{\partial w}{\partial z}(z) + (z^2 - p(p+N-2)) w(z) = 0,
\end{equation*}
and it is easy to see that
\begin{equation*}
j_p(x) = x^{1 - \frac{N}{2}} J_{p + \frac{N}{2} - 1}(x), \quad y_p(x) = x^{1 - \frac{N}{2}} Y_{p + \frac{N}{2} - 1}(x).
\end{equation*}
In general the dependence upon the dimension $N$ is not explicitly specified, and canonical Bessel functions (usually referred to as cylindrical) are the ultraspherical functions for $N = 2$. Notice that the expansion of $y_p$ for $N$ even can be easily deduced from that of $Y_p$, while for $N \geq 3$ odd and $p \in \mathbb{N}$ we have
\begin{equation*}
y_p(x) = (-1)^{p + \frac{N-1}{2}} 2^{p + \frac{N}{2} - 1} \sum_{k=0}^{\left\lfloor \frac{p+N-3}{2} \right\rfloor} \frac{x^{2k - p - N + 2}}{k! (-4)^k \Gamma(k - p - \frac{N}{2} + 2)} + P_p(x),
\end{equation*}
where $\left\lfloor x \right\rfloor$ is the integer part of $x$ and $P_p$ is a $C^\infty$-function.

It is also possible to express the derivatives (with respect to $x$) of Bessel functions in terms of other Bessel functions. From \cite[Eq. 10.6.2]{olvernist} it is easy to deduce that
\begin{equation*}
\begin{split}
    j'_p(x) &= -j_{p+1}(x) + \frac{p}{x} j_p(x) = j_{p-1}(x) - \frac{p + N - 2}{x} j_p(x),
    \\
    j''_p(x) &= \frac{N-1}{x} j_{p+1}(x) + \frac{p^2 - p - x^2}{x^2} j_p(x),
\end{split}
\end{equation*}
and
\begin{equation*}
\begin{split}
y'_p(x) &= -y_{p+1}(x) + \frac{p}{x} y_p(x) = y_{p-1}(x) - \frac{p + N - 2}{x} y_p(x),
\\
y''_p(x) &= \frac{N-1}{x} y_{p+1}(x) + \frac{p^2 - p - x^2}{x^2} y_p(x).
\end{split}
\end{equation*}
Analogously, the modified Bessel functions of the first kind $I_p$ and of the second kind $K_p$ are the solutions to modified Bessel’s equation
\begin{equation*}
z^2 \frac{\partial^2 w}{\partial z^2}(z) + z \frac{\partial w}{\partial z}(z) - (z^2 + p^2) w(z) = 0.
\end{equation*}
We observe that it is possible to obtain the modified Bessel’s equation from the canonical Bessel’s equation replacing $z$ by $\pm iz$. Modified Bessel functions admit the following series representations for real arguments:
\begin{equation*}
\begin{split}
I_p(x) &= \frac{x^p}{2^p} \sum_{k=0}^{\infty} \frac{1}{k! \Gamma(p + k + 1)} \left(\frac{x^2}{4}\right)^k,
\\
K_p(x) &= \frac{\pi}{2} \frac{I_{-p}(x) - I_p(x)}{\sin(p \pi)}, \quad \text{if } p \notin \mathbb{Z},
\\
K_p(x) &= \frac{(-1)^{p-1}}{2} \frac{\partial I_q(x)}{\partial q} \bigg|_{q=p} + \frac{(-1)^{p-1}}{2} \frac{\partial I_q(x)}{\partial q} \bigg|_{q=-p}, \quad \text{if } p \in \mathbb{Z}.
\end{split}
\end{equation*}
In particular,
\begin{equation*}
K_p(x) = \frac{2^{p-1}}{x^p} \sum_{k=0}^{p-1} \frac{(p-k-1)! x^{2k}}{k! (-4)^k} + (-1)^{p+1} \ln \left( \frac{x}{2} \right) I_p(x) + \tilde{Q}_p(x), \quad \text{if } p \in \mathbb{N} \setminus \{0\},
\end{equation*}
and 
\begin{equation*}
K_0(x) = -\left( \gamma_{\text{EM}} + \ln \left( \frac{x}{2} \right) \right) I_0(x) + \tilde{Q}_0(x),
\end{equation*}
where $\gamma_{\text{EM}}$ is the Euler-Mascheroni constant and $\tilde{Q}_p$ are $C^\infty$-functions (see \cite[Sec. 10.31]{olvernist} for more details).

The modified ultraspherical Bessel functions $i_p, k_p$ are defined as the solutions to the modified ultraspherical Bessel equation
\begin{equation*}
z^2 \frac{\partial^2 w}{\partial z^2}(z) + z(N-1) \frac{\partial w}{\partial z}(z) - (z^2 + p(p + N-2)) w(z) = 0,
\end{equation*}
and it is easy to see that
\begin{equation*}
i_p(x) = x^{1 - \frac{N}{2}} I_{p + \frac{N}{2} - 1}(x), \quad k_p(x) = x^{1 - \frac{N}{2}} K_{p + \frac{N}{2} - 1}(x).
\end{equation*}
In general the dependence upon the dimension $N$ is not explicitly specified, and canonical modified Bessel functions are the ultraspherical functions for $N = 2$. Notice that the expansion of $k_p$ for $N$ even can be easily deduced from that of $K_p$, while for $N \geq 3$ odd and $p \in \mathbb{N}$ we have
\begin{equation*}
k_p(x) = (-1)^{p + \frac{N + 1}{2}} \pi 2^{p + \frac{N}{2} - 2} \sum_{k=0}^{\left\lfloor \frac{p+N-3}{2} \right\rfloor} \frac{x^{2k - p - N + 2}}{k! (4)^k \Gamma(k - p - \frac{N}{2} + 2)} + \tilde{P}_p(x),
\end{equation*}
where again $\left\lfloor x \right\rfloor$ is the integer part of $x$ and $\tilde{P}_p$ is a $C^\infty$-function.

It is also possible to express the derivatives (with respect to $x$) of modified Bessel functions in terms of other modified Bessel functions. From \cite[Eq. 10.29.2]{olvernist} it is easy to deduce that
\begin{equation*}
\begin{split}
i'_p(x) &= i_{p+1}(x) + \frac{p}{x} i_p(x) = i_{p-1}(x) - \frac{p + N - 2}{x} i_p(x),
\\
i''_p(x) &= -\frac{N-1}{x} i_{p+1}(x) + \frac{p^2 - p + x^2}{x^2} i_p(x),    
\end{split}
\end{equation*}
and 
\begin{equation*}
\begin{split}
k'_p(x) &= -k_{p+1}(x) + \frac{p}{x} k_p(x) = -k_{p-1}(x) - \frac{p + N - 2}{x} k_p(x),
\\
k''_p(x) &= \frac{N-1}{x} k_{p+1}(x) + \frac{p^2 - p + x^2}{x^2} k_p(x).
\end{split}
\end{equation*}

\section*{Acknowledgments}
The authors are members of the Gruppo Nazionale per l'Analisi Matematica, la Probabilità e le loro
Applicazioni (GNAMPA) of the Istituto Nazionale di Alta Matematica (INdAM). We acknowledge support from the project "Perturbation problems and asymptotics for elliptic differential equations: variational and potential theoretic methods" funded by the European Union - Next Generation EU and by MUR Progetti di Ricerca di Rilevante Interesse Nazionale (PRIN) Bando 2022 grant 2022SENJZ3.


\end{document}